\documentclass{amsart}

\usepackage{amssymb}
\usepackage{amsthm}
\usepackage{amsmath}
\usepackage[dvipdfmx]{graphicx}
\usepackage{txfonts}
\usepackage[all]{xy}
\usepackage{color}

\theoremstyle{plain}
\newtheorem{theorem}{Theorem}[section]
\newtheorem{proposition}[theorem]{Proposition}
\newtheorem{corollary}[theorem]{Corollary}
\newtheorem{lemma}[theorem]{Lemma}

\newtheorem*{Theorem}{Theorem}

\newtheorem*{conjecture*}{Conjecture}
\newtheorem*{thm:RuelleRtorsion}{Theorem~\ref{thm:Ruelle_Rtorsion_eqn}}
\theoremstyle{definition}
\newtheorem{definition}[theorem]{Definition}
\newtheorem*{definition*}{Definition}
\newtheorem{example}[theorem]{Example}
\newtheorem*{example*}{Example}
\newtheorem*{notation*}{Notation}
\newtheorem*{notation-conv*}{Notation and convention}
\newtheorem*{convention*}{Convention}
\theoremstyle{remark}
\newtheorem{remark}[theorem]{Remark}
\newtheorem*{remark*}{Remark}

\newcommand{\Z}{\mathbb{Z}}

\newcommand{\C}{\mathbb{C}}

\newcommand{\SL}{\mathrm{SL}_2(\C)}

\newcommand{\SU}{\mathrm{SU}(2)}

\newcommand{\trace}{\mathop{\mathrm{tr}}\nolimits}

\newcommand{\bm}[1]{\mbox{\boldmath{$#1$}}}

\newcommand{\bd}{\partial}

\newcommand{\Tor}[1]{\mathop{\mathbb{T}_{\mu}}\nolimits (#1)}

\begin{document}

\title[]{Gang-Kim-Yoon integrality conjectures on adjoint Reidemeister torsions for  torus knots
}

\author{Yuji Terashima}
\address{Graduate school of science, Tohoku University,
6-3, Aoba, Aramaki-aza, Aoba-ku, Sendai, 980-8578, Japan}
\email{yujiterashima@tohoku.ac.jp}

\author{Yoshikazu Yamaguchi}
\address{Faculty of Commerce,
  Waseda University,
  1-6-1 Nishiwaseda, Shinjuku-ku,
  Tokyo, 169-8050, Japan}
\email{shouji@waseda.jp}


\begin{abstract}
  We study the conjecture that a sum of the $(g-1)$-st powers of adjoint Reidemeister torsions for a torus knot is an integer. We prove that the conjecture is true for any torus knot and all $g \geq 0$.
  To prove the conjecture, we introduce the Verlinde numbers for torus knots from the viewpoint of modular $S$-matrix and show the recursion formulas and initial values of them. The recursion formulas of Verlinde numbers prove the integrality of the sum of the $(g-1)$-st powers of adjoint Reidemeister torsions.
  Related to a modular $S$-matrix, we also provide a birational model of the character variety for a torus knot and show how to recover the adjoint Reidemeister torsion for a torus knot from the Hessian of the polynomial defining the birational model.
\end{abstract}

\maketitle

\section{Introduction}
Recently, Gang, Kim, and Yoon \cite{GKY} proposed a striking conjecture stating that, for a $3$-manifold and any non-negative integer $g$, a sum of the $(g-1)$-st powers of adjoint Reidemeister torsions for irreducible representations of the fundamental group is an integer. This conjecture is based on the $3$d–$3$d correspondence \cite{TY1,TY2,DG, DGG,CCV}, a physical duality that relates geometry of a $3$-manifold $M$ to $3$-dimensional gauge theory labeled by $M$. More concretely, via the $3$d–$3$d correspondence, they identify the index of a 3-dimensional gauge theory on the product of a genus $g$ surface and a circle with the sum of adjoint Reidemeister torsions above, and derive the conjecture from the integrality of the index. See also \cite{GK, GKP, BGP}.

In this paper, we prove this integrality conjecture for the case of complements of torus knots, for arbitrary non-negative integers $g$. The key idea of the proof is to introduce modular $S$-matrices for torus knots, formulated as generalizations of the modular $S$-matrices of Wess-Zumino-Witten models, and to interpret the sum of adjoint Reidemeister torsions above as a generalization of the Verlinde numbers of Wess-Zumino-Witten models. We then prove the conjecture as follows. 
\begin{Theorem}[Theorem~\ref{thm:integerness}]
  The following sum of the adjoint Reidemeister torsions for any $(p,\, q)$-torus knot exterior on a level set of the trace function of the meridian $\mu$ at a generic $c \in \C$ 
  \[\sum_{[\rho] \in \mathrm{tr}_\mu^{-1}(c)} (2\Tor{\rho})^{g-1}\]  
  is an integer for all $g \geq 0$.
\end{Theorem}
This theorem is derived by showing that this generalized Verlinde number satisfies a recursion relation with respect to the genus $g$. The integrality of the initial value in the recursion follows from a result of \cite{TranY}; in fact, we show that for all torus knots, each sequence $\{\sum_{[\rho] \in \mathrm{tr}_\mu^{-1}(c)} (2\Tor{\rho})^{g-1}\}$ on $g$ starts with the common value $1$ at $g=0$, that is, it holds that $\sum_{[\rho] \in \mathrm{tr}_\mu^{-1}(c)} (2\Tor{\rho})^{-1}=1$ for all torus knots. 
After the first version of this paper was posted on arXiv, it was pointed out by A. T. Tran that this result had been obtained in the preprint \cite{MorifujiTran} independently of our work.

It is expected that the modular $S$-matrix above for torus knots will play an important role in research that links mathematics of knot theory with physics of gauge theory such as a project to associate modular tensor categories to 3-manifolds \cite{CGK, CQW}, knot-quiver correspondence \cite{KRSS1,KRSS2,EKL} and knot-gauge correspondence \cite{MTT}. 

In this paper, we focus on $S$-matrices derived from plane algebraic curves which is called \textit{Chebyshev curves}.
We propose a rational model of the character varieties of torus knot groups as the resolution of Chebyshev curves (see Theorem~\ref{thm:birational}) and recover the adjoint Reidemeister torsions of a torus knot from the Hessians of the defining polynomial of a Chebyshev curve (see Theorem~\ref{thm:RtorsionHessian}) at a critical point as follows.

Set $F(X, Y) = C_p(X) - C_q(Y)$ for coprime integers $p$ and $q$ where $C_k(z)$ is the Chebyshev polynomial of the first kind (see SubSect.~\ref{sec:chebyshev} for the details) and denote by $C_{p,\, q}$ the Chebyshev curve $F(X, Y)=0$.
\begin{Theorem}[Theorem~\ref{thm:birational}]
 The resolution $Y_{p,\, q}$ of $C_{p,\, q}$ by blowing-up consists of a non-singular curve $D_{p,\, q}$ and $(p-1)(q-1)/2$ exceptional lines in $\C^2 \times \C$.

  The character variety $X(G_{p,\, q})$ of a torus knot group $G_{p,\, q}$ is birational to $Y_{p,\, q}$.
  More precisely, the component $X^{red}(G_{p,\, q})$ of reducible characters is birational to $D_{p,\, q}$ and the others $X^{irr}(G_{p,\, q})$ of irreducible characters are isomorphic to $Y_{p,\, q} \setminus D_{p,\, q}$.
\end{Theorem}
The $Y_{p,\, q} \setminus D_{p,\, q}$ above is the union of $L_{a,\, b} = \C \setminus \{\hbox{two points}\}$ ($0<a<p$, $0<b<q$, $a \equiv b \pmod{2}$). Each component $L_{a, b}$ lies on the critical point $(2\cos(\pi a / p),\, 2\cos(\pi b / q))$ of $F(X,\, Y)$ over the Chebyshev curve $C_{p,\, q}$. The adjoint Reidemeister torsion gives the following locally constant function on $Y_{p,\, q} \setminus D_{p,\, q}$.
\begin{Theorem}[Theorem~\ref{thm:RtorsionHessian}]
  The adjoint Reidemeister torsion $\Tor{\rho}$ for the $(p,\, q)$-torus knot exterior and a meridian $\mu$ satisfies
  \[
    \Tor{\rho}
    = \frac{-1}{4pq}\det
    \left.
      \left(\frac{\partial (F_X,\, F_Y)}{\partial (X,\, Y)} \right)
      \right|_{(X,\, Y)=(2\cos (a\pi/p),\, 2\cos (b\pi/q))}
   \]
   where $(\partial (F_X,\, F_Y) / (X,\, Y))$ is the Jacobian determinant of partial derivative $F_X$ and $F_Y$ of $F(X, Y)$, which are the defining polynomials of exceptional lines, on $L_{a,\, b}$.
   In other words, it holds that
   \[
   \Tor{\rho}= \frac{-1}{4pq} \left.\mathrm{Hess}(F)\right|_{(X,\, Y)=(2\cos(\pi a / p),\, 2\cos(\pi b / q))}
   \]
  on $L_{a,\, b}$.
  Here $\mathrm{Hess}(F)$ is the Hessian of the defining polynomial $F(X, Y)$ of the Chebyshev curve $C_{p,\, q}$.
\end{Theorem}

\medskip

We also touch recent developments around Gang-Kim-Yoon's integrality conjectures.
As for the case of $g=0$ and a compact hyperbolic $3$-manifold with torus boundary, it is conjectured that the inverse sum of adjoint Reidemeister torsions equals zero. This is called a \textit{vanishing conjecture} on adjoint Reidemeister torsions. We can find supporting evidence of vanishing conjecture for the knot exteriors of the figure--eight knot~\cite{GKY}, hyperbolic twist knots~\cite{Yoon} and hyperbolic two--bridge knots~\cite{Yoon2}. As other supporting examples,
Tran and the second author showed that the vanishing identity also holds for once--punctured torus bundle over the circle with tunnel number one.
There is another different perspective in which N.~Wakijo~\cite{Wakijo} studied the inverse sum of adjoint Reidemeister torsions for closed $3$-manifolds obtained by $p/1$ and $1/q$-surgeries along the figure--eight knot. 

\medskip

This paper is organized as follows:
Sect.~2 reviews the character varieties of torus knot groups and provides a rational model for the character varieties of torus knot groups by a resolution of Chebyshev curves. 
We will see the Hessian of the defining polynomial for a Chebyshev curve gives the adjoint Reidemeister torsions for a torus knot exterior in Sect.~3.
Sect.~4 introduces the Verlinde numbers of torus knots. Sect.~5 gives the recursion relation and the initial values of our Verlinde numbers and shows the integrality of the power sum of the adjoint Reidemeister torsions. We also touch a relation between the generalized Verlinde numbers and classical ones in Sect.~6.

\subsection*{Acknowledgements}
  We would like to express our gratitude to H.~Fuji, D.~Gang, M.~Manabe, R.~Tange, and S.~Terashima for helpful discussions, and to T.~Nosaka for carefully reading the first version of this paper and providing many useful comments. We are also grateful to A.~T.~Tran for kindly informing us of an interesting, relevant preprint.
  The authors are supported by JSPS KAKENHI Grant Number 22H01117, 24K06720 and 25K06969.
  
\section{Character varieties for torus knots}
We start with a review on some known results on the character varieties of torus knot groups and then we will provide a rational model of our character varieties by resolutions of algebraic plane curves, called \textit{Chebyshev curves}, in $\C^2 \times \mathbb{P}^1$.
\subsection{Review on the character variety for a torus knot group}

Let $p$ and $q$ be coprime integers. Let $T_{p,\, q}$ be the torus knot of type $(p,\, q)$.
For simplicity, we assume that $p$ and $q$ are positive and $q$ is always odd.
We write $E_{p,\, q}$ for the torus knot exterior $S^{3} \setminus N(T_{p,\, q})$ in the $3$-sphere.
Here $N(T_{p,\, q})$ denotes an open tubular neighbourhood of the knot $T_{p,\, q}$.

It is known that the knot group $G_{p,\, q}=\pi_1(E_{p,\, q})$ has the following presentation:
\begin{equation}
  \label{eq:pres_torusknotgr}
  G_{p,\, q} = \langle \alpha, \beta \,|\, \alpha^p = \beta^q \rangle
\end{equation}
A meridian $\mu$ and a preferred longitude $\lambda$ of $T_{p,\, q}$ are expressed as
\begin{equation}
  \mu = \alpha^{-r} \beta^s \quad \text{and} \quad \lambda = \alpha^p \mu^{-pq} (=\beta^{q}\mu^{-pq}) 
\end{equation}
for integers $r$ and $s$ such that $ps-qr=1$ in the presentation~\eqref{eq:pres_torusknotgr}.
$\SL$-representations $\rho$ of $G_{p,\, q}$ mean homomorphisms from $G_{p,\, q}$ into $\SL$. The \textit{character} of an $\SL$-representation $\rho$ is defined as a function $\chi_\rho : G_{p,\, q} \to \C$ given by $\chi_\rho(g) = \mathrm{tr}\, \rho(g)$.
The set of characters has a structure of affine algebraic variety, which is called \textit{the character variety} of the torus knot group $G_{p,\, q}$ and denoted by $X(G_{p,\, q})$.

An $\SL$-representation $\rho$ is referred to as being \textit{irreducible} if there is no invariant line in $\C^2$ under the action of $\rho(G_{p,\, q})$ and $\rho$ is said to be \textit{reducible} otherwise.
We write $X^{irr}(G_{p,\, q})$ and $X^{red}(G_{p,\, q})$ for the set of characters of irreducible and reducible $\SL$-representations of $G_{p,\, q}$ respectively. 

It is shown in~\cite{CS} that irreducible $\SL$-representations $\rho$ and $\rho'$ are conjugate if and only if $\chi_\rho = \chi_{\rho'}$.
We can regard $X^{irr}(G_{p,\, q})$ as the set of conjugacy classes $[\rho]$ of irreducible $\SL$-representations $\rho$ of $G_{p,\, q}$.
The set $X^{irr}(G_{p,\, q})$ is the disjoint union of algebraic varieties which are isomorphic to the complex lines $\C$ without two points.
We call $X^{irr}(G_{p,\, q})$ the irreducible part of the character variety $X(G_{p,\, q})$. According to~\cite{Johnson}, the irreducible part $X^{irr}(G_{p,\, q})$ consists of $(p-1)(q-1)/2$ components as follows. 
\begin{itemize}
\item the traces $\trace \rho(\alpha)$ and $\trace \rho(\beta)$ have the constant values
  $2 \cos (\pi a / p)$ and $2 \cos (\pi b / q)$ such that $a \equiv b \pmod{2}$ and $0 < a < p$, $0 < b < q$ on each component in $X(G_{p,\, q})$.
\item the trace $\trace \rho(\mu)$ on every component gives an isomorphism to $\C \setminus \{2\cos(ar\pi/p \pm bs\pi/q)\}$ (i.e., except for two points).
\end{itemize}
We write $\C_{a, b}$ for each component of $X^{irr}(G_{p,\, q})$ such that $\trace \rho(\alpha)=2\cos(\pi a / p)$ and $\trace \rho(\beta) = 2\cos(\pi b/q)$ for all conjugacy classes $[\rho]$ on the component.
\begin{example}
  We examine the images of $\alpha$, $\beta$ and $\mu$ in $G_{p,\, q}$ under an irreducible $\SL$-representation $\rho$.
  We can assume that $\rho(\alpha)$, $\rho(\beta)$ are given by
  \[\rho(\alpha)
    = \begin{pmatrix}
      x & 1 \\
      0 & x^{-1}
    \end{pmatrix} \quad \text{and}\quad
    \rho(\beta)
    = \begin{pmatrix}
      y & 0 \\
      u & y^{-1}
    \end{pmatrix}
  \]
  up to conjugation.
  The relation $\alpha^p = \beta^q$ gives the equality that
  \[
    \begin{pmatrix}
      x^p & (x^p - x^{-p})/(x-x^{-1}) \\
      0 & x^{-p}
    \end{pmatrix}
    =\begin{pmatrix}
      y^q & 0 \\
      u(y^q - y^{-q})/(y-y^{-1}) & y^{-q}
    \end{pmatrix}.
  \]
  Thus we have $x=e^{\pm a\pi\sqrt{-1}/p}$ and $y=e^{\pm b \pi\sqrt{-1}/q}$ where $0 < a < p$, $0 < b < q$ and $a \equiv b \pmod{2}$.
   If we set $x=e^{a\pi\sqrt{-1}/p}$ and $y=e^{b\pi\sqrt{-1}/q}$, then the trace of $\rho(\mu) = \rho(\alpha^{-r}\beta^s)$ is given by
  \begin{align*}
    &x^{-r} y^s + x^r y^{-s} - u \frac{x^r - x^{-r}}{x-x^{-1}} \frac{y^s - y^{-s}}{y-y^{-1}} \\
    &= 2 \cos (-ar\pi/p + bs\pi/q) - u \frac{\sin (ar\pi/p)}{\sin (a\pi/p)} \frac{\sin (bs\pi/q)}{\sin (b\pi/q)}.
  \end{align*}
  We need to exclude $u=0$ and $u=4\sin(a\pi/p)\sin(b\pi/q)$ for $\rho$ to be irreducible.
  In these cases, $\mathrm{tr}\,\rho(\mu)$ turns out to be $2 \cos(ar\pi/p \mp bs\pi/q)$.
\end{example}

We also set the function $\mathrm{tr}_\gamma$ for any $\gamma \in G_{p,\, q}$ on $X(G_{p,\, q})$ as
\[\mathrm{tr}_\gamma(\chi_{\rho}) = \trace \rho(\gamma) \,( =\chi_{\rho}(\gamma) )\]
and call it the \textit{trace function} of $\gamma$.

On the reducible part $X^{red}(G_{p,\, q})$, each character $\chi_{\rho}$ factors through the abelianization $\mathrm{ab}: G_{p,\, q} \to G_{p,\, q} / [G_{p,\, q}, G_{p,\, q}] \simeq H_1(E_{p,\, q}; \Z) = \langle [\mu] \rangle$. We write $\varphi_\rho : \langle \mu \rangle \to \C$ for the induced function from the character of a reducible $\SL$-representation $\rho$ such that $\chi_\rho = \varphi_\rho \circ \mathrm{ab}$. Since $\varphi_\rho$ is determined by the image of $\mu$, the character $\chi_\rho$ is also determined by the image of $\mu$. Therefore the trace function of $\mu$ also gives an isomorphism between the reducible part $X^{red}(G_{p,\, q})$ and $\C$. We will see explicit forms of trace functions on the reducible part in the following subsections and give a birational model of algebraic curve. 

\subsection{Chebyshev curves}
\label{sec:chebyshev}
We denote by $C_k(z)$ and $S_k(z)$ $(k \in \Z)$ the Chebyshev polynomials of the first and second kinds such that
\begin{gather*}
  C_{k+1}(z) = zC_k(z) - C_{k-1}(z), \quad C_1(z)=z,\, C_0(z)=2, \\
  S_{k+1}(z) = zS_k(z) - S_{k-1}(z), \quad S_1(z)=z,\, S_0(z)=1.
\end{gather*}
\begin{remark}
  The Chebyshev polynomials $T_k(z)$ and $U_k(z)$ of the first and second kinds are usually defined by
  \begin{gather*}
  T_{k+1}(z) = 2z T_k(z) - T_{k-1}(z), \quad T_1(z)=z,\, T_0(z)=1, \\
  U_{k+1}(z) = 2z U_k(z) - U_{k-1}(z), \quad U_1(z)=2z,\, U_0(z)=1.
\end{gather*}
Our Chebyshev polynomials satisfy that
\[C_k(z) = 2T_k(z/2) \quad \text{and} \quad S_k(z) = U_k(z/2).\]
Moreover if $ z = \zeta + \zeta^{-1}$, then $C_k(z)$ and $S_k(z)$ are expressed as
\[C_k(\zeta + \zeta^{-1}) = \zeta^k + \zeta^{-k} \quad \text{and} \quad
  S_k(\zeta + \zeta^{-1}) = \frac{\zeta^{k+1} - \zeta^{-(k+1)}}{\zeta - \zeta^{-1}}.\]
\end{remark}
\begin{lemma}
  \label{lem:rationality_Xred}
  If the homology class of $\gamma \in G_{p,\, q}$ is expressed as $[\gamma] = [\mu^k]$, then the trace function $\mathrm{tr}_\gamma$ satisfies $C_k \circ \mathrm{tr}_{\mu}$ on the reducible part $X^{red}(G_{p,\, q})$.

  In particular, the trace functions $\mathrm{tr}_{\alpha}$ and $\mathrm{tr}_{\beta}$ of the generators $\alpha$ and $\beta$ are expressed as $\mathrm{tr}_{\alpha} = C_q \circ \mathrm{tr}_{\mu}$ and $\mathrm{tr}_{\beta} = C_p \circ \mathrm{tr}_{\mu}$.
\end{lemma}
\begin{proof}
  When we set $\chi_\rho(\mu) = m + m^{-1}$, the value $\chi_\rho(\mu^k)$ is given by $C_k(m+m^{-1})$. The character $\chi_\rho$ of a reducible $\SL$-representation factors through $H_1(E_{p,\, q}; \Z)=\langle [\mu] \rangle$ since we can assume the image of $\rho$ is contained in upper triangular matrices in $\SL$. The character $\chi_\rho$ sends $\gamma$ to $C_k(m+m^{-1})$ for any $\gamma \in G_{p,\, q}$ in the homology class  $[\mu^k]$.

  It is enough to show that the homology classes $[\alpha]$ and $[\beta]$ are given by $[\mu^q]$ and $[\mu^p]$ for $\mathrm{tr}_{\alpha} = C_q \circ \mathrm{tr}_{\mu}$ and $\mathrm{tr}_{\beta} = C_p \circ \mathrm{tr}_{\mu}$.
  The relation $\alpha^p = \beta^q$ implies $[\alpha]=[\mu^{kq}]$ and $[\beta]=[\mu^{kp}]$. By $\mu=\alpha^{-r}\beta^{s}$ for $ps-qr=1$ we have $k=1$.
\end{proof}

We review a notion of \textit{Chebyshev curves}, which are plane algebraic curves defined by two Chebyshev polynomials.
\begin{definition}[Chebyshev curve of type $(p,\, q)$]
  For integers $p$ and $q$, we define the Chebyshev curve $C_{p,\, q}$ of type $(p,\, q)$ on $\C^2$ as
  \[C_{p,\, q} = \{(X, Y) \in \C^2 \,|\, C_p(X) - C_q(Y) = 0\}\]
\end{definition}
\begin{figure}[ht]
  \includegraphics[width=10cm]{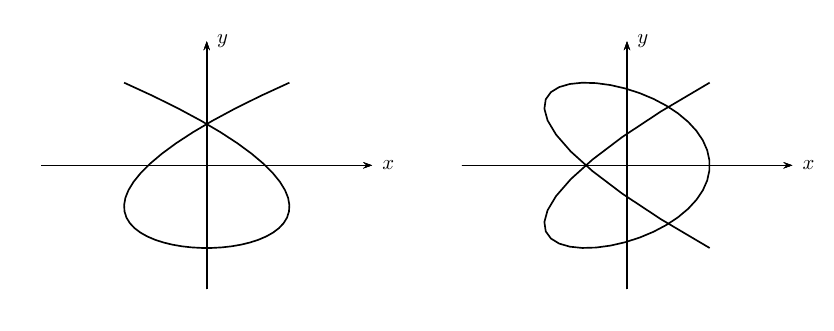}
  \caption{Chebyshev curves of type $(2, 3)$ and $(3, 4)$}
\end{figure}
We can find fundamental properties of Chebyshev curves in~\cite[Sec.~3.9]{Fis}.
\begin{remark}
  If $p$ and $q$ are coprime, the Chebyshev curve $C_{p,\, q}$ is irreducible.
\end{remark}
For coprime $p$ and $q$, the Chebyshev curve has $(p-1)(q-1)/2$ nodes (double points) as singularities. These singular points are given by 
\[
\left\{\left.
      \left(2\cos \frac{a \pi}{p}, 2\cos \frac{b\pi}{q} \right) \,\right|\,
      0<a<p, \quad 0<b<q, \quad  a \equiv b \pmod{2}
\right\}
\]
from the system of $F(X, Y)=0$, $F_X(X, Y)=0$ and $F_Y(X, Y)=0$ where $F(X, Y)=C_p(X)-C_q(Y)$.
For the details, we refer the reader to~\cite[Sec.~3.9]{Fis}.
We denote by $\mathrm{Sing}(C_{p,\, q})$ the set of singular points above.

The solutions of the system $F_X(X, Y)=0$ and $F_Y(X, Y)=0$ as above give $(p-1)(q-1)$ points $\{(2\cos a \pi / p, 2\cos b \pi / q) \,|\, 0<a<p, 0<b<q \}$. 
We blow up $\C^2$ on these $(p-1)(q-1)$ points and give a resolution of singularity for the Chebyshev curve $C_{p,\, q}$.

We write $B$ for the resulting surface by the blowing-up of $\C^2$ on $(p-1)(q-1)$ points $\{(2\cos a \pi / p, 2\cos b \pi / q) \,|\, 0<a<p,\, 0<b<q \}$ and $D_{p,\, q}$ for the resolution of the Chebyshev curve $C_{p,\, q}$. Summarizing,
we can express the resulting surface $B$ and the non-singular curve $D_{p,\, q}$ as 
\[ B = \{(X, Y, [Z_0 : Z_1]) \in \C^2 \times \mathbb{P}^1 \,|\, F_X(X, Y)Z_0 = -F_Y(X, Y)Z_1 \}\]
where $F(X, Y) = C_p(X) -C_q(Y)$ and 
\[ D_{p,\, q} = \overline{\pi^{-1} ( C_{p,\, q} \setminus \mathrm{Sing}(C_{p,\, q}))} \]
which is the closure of $\pi^{-1} ( C_{p,\, q} \setminus \mathrm{Sing}(C_{p,\, q}))$ by the projection $\pi: B \to \C^2$.
We can realize the closure of $\pi^{-1} ( C_{p,\, q} \setminus \mathrm{Sing}(C_{p,\, q}))$ by a parametrization of the Chebyshev curve as follows.
\begin{proposition}
  \label{prop:param_Dpq}
  The algebraic curve $D_{p,\, q}$ is parametrized as
  \[ D_{p,\, q} = \{(C_q(t), C_p(t), [C'_q(t) : C'_p(t)]) \in \C^2 \times \mathbb{P}^1\,|\, t \in \C\}.\]
  In particular, $D_{p,\, q}$ intersects with the exceptional line on $(2\cos (a\pi/p), 2\cos(b\pi/q))$ at $[q \sin(a\pi/p) : \pm p \sin (b\pi/q)] \in \mathbb{P}$.
\end{proposition}
\begin{proof}
  The Chebyshev polynomials $C_k(z)$ satisfy $C_k(C_l(z)) = C_{kl}(z)$.
  The defining equation $F(X, Y)=C_p(X)-C_q(Y)=0$ of the Chebyshev curve $C_{p,\, q}$ has the solution $X=C_q(t)$ and $Y=C_p(t)$. Moreover we have $[Z_0 : Z_1] = [C'_q(t):C'_p(t)]$ by 
  \[F_X(C_q(t), C_p(t))C'_q(t)+F_Y(C_q(t), C_p(t))C'_p(t)=0\]
  on $C_{p,\, q}$.

  By next Lemma $(C_q(t), C_p(t))=(2\cos(a\pi/p), 2\cos(b\pi/q))$ implies $t=2\cos(ar\pi/p \pm bs\pi/q)$ where $a \equiv b \pmod{2}$ and  $ps-qr=1$. It follows from $C'_k(z) = k S_{k-1}(z)$ that
  \begin{align*}
    [C'_q(t) : C'_p(t)]
    &= \left[q \frac{\sin(aqr\pi/p \pm bs\pi)}{\sin (ar\pi/p \pm bs\pi / q)}
      : p \frac{\sin (ar\pi \pm bps\pi/q) }{\sin (ar\pi/p \pm bs\pi/q)}\right] \\
    &= [q\sin(a\pi/p) : \mp p \sin (b\pi/q)]
  \end{align*}
\end{proof}
\begin{lemma}
  The system of $C_q(t)=2\cos(a\pi/p)$ and $C_p(t)=2\cos(b\pi/q)$ for $a \equiv b \pmod{2}$ has the solution $t=2\cos (ar\pi/p \pm bs\pi/q)$ with $ps-qr=1$.
\end{lemma}
\begin{proof}
  If we set $t=m+m^{-1}$, then we have $m = e^{\pm a\pi \sqrt{-1}/(pq) + 2\pi k \sqrt{-1}/q}$ with some $k \in \Z$ by
  \[C_q(t) = e^{a\pi \sqrt{-1}/p} + e^{-a\pi \sqrt{-1}/p}.\]
  Similarly we have $m = e^{\pm b\pi \sqrt{-1}/(pq) + 2\pi \ell \sqrt{-1}/p}$ with some $\ell \in \Z$ by $C_p(t) = 2\cos(b\pi/q)$.
  By $a/(pq) + 2k/q = \pm b /(pq) + 2\ell/p$ it holds that
  $kp - \ell q = (-a \pm b)/2$.
  We can set $k=s(-a \pm b)/2$ and $\ell = r(-a \pm b)/2$ by $ps-qr=1$,
  which implies $t=2\cos (ar\pi/p \pm bs\pi/q)$.
  Note that we have the same $t$ if we replace $s$ and $r$ with $s + dq$ and $r + dp$.
\end{proof}

\subsection{Birational model of character varieties for torus knots}
We provide a birational model of the character variety $X(G_{p,\, q})$ through blowing-up the Chebyshev curve $C_{p,\, q}$ at $(p-1)(q-1)/2$ points.
We use the same symbol $\pi$ for the restriction of $\pi : B (\subset \C^2 \times \mathbb{P}^1) \to \C^2$ to $B \cap \C^2 \times (\mathbb{P}^1 \setminus \{\infty = [1:0]\})$.
We set $Y_{p,\, q} = \pi^{-1}(C_{p,\, q})$ in $\C^2 \times (\mathbb{P}^1 \setminus \{\infty = [1:0]\})$. Our variety $Y_{p,\, q}$ is the union of the non-singular curve $D_{p,\, q}$ and $(p-1)(q-1)/2$ exceptional lines. We have seen in Proposition~\ref{prop:param_Dpq} that each exceptional line intersects with $D_{p,\, q}$ at two points. We will show that the decomposition $Y_{p,\, q} = D_{p,\, q} \cup (Y_{p,\, q} \setminus D_{p,\, q})$ provides a birational model of $X(G_{p,\, q})=X^{red}(G) \cup X^{irr}(G_{p,\, q})$.
\begin{theorem}
  \label{thm:birational}
  The character variety $X(G_{p,\, q})$ is birational to $Y_{p,\, q}$.
  More precisely, the reducible part $X^{red}(G_{p,\, q})$ is birational to $D_{p,\, q}$ and the irreducible part $X^{irr}(G_{p,\, q})$ is isomorphic to $Y_{p,\, q} \setminus D_{p,\, q}$.
\end{theorem}
\begin{proof}
  We show that both of $X^{red}(G_{p,\, q})$ and $D_{p,\, q}$ are rational curves.
  Here a rational curve means that the function field of the curve is the field of rational functions in one indeterminate.
  Set $t=\mathrm{tr}_{\mu}$.
  Lemma~\ref{lem:rationality_Xred} implies that the function field of $X^{red}(G_{p,\, q})$ is $\C(t)$. Thus $X^{red}(G_{p,\, q})$ is a rational curve.
  It is known that an algebraic curve has a parametrization if and only if it is a rational curve.
  We can see that $D_{p,\, q}$ is also a rational curve by the parametrization of $D_{p,\, q}$ in Proposition~\ref{prop:param_Dpq} 

  Actually we define a map $\Phi^{red}: X^{red}(G_{p,\, q}) \to D_{p,\, q}$ by the triple of the trace functions
  $\mathrm{tr}_\alpha = C_q(t)$, $\mathrm{tr}_\beta = C_p(t)$ and $\mathrm{tr}_\mu = t$ as
  \[\Phi^{red}(\chi_\rho) = (\mathrm{tr}_\alpha, \mathrm{tr}_\beta, [C'_q(t):C'_p(t)]) = (C_q(t), C_p(t), [C'_q(t):C'_p(t)])\]
  under $C'_q(t) \not = 0$.
  We can regard $\Phi^{red}$ as a birational morphism from $X^{red}(G_{p,\, q})$ to $D_{p,\, q}$.

  Set $L_{a, b} = \pi^{-1}((2\cos(a\pi/p), 2\cos(b\pi/q))) \setminus D_{p,\, q}$.
  According to Proposition~\ref{prop:param_Dpq}, $L_{a, b}$ is expressed as $L_{a, b} = \mathbb{P}^1 \setminus \{[q \sin(a\pi/p) : \pm p \sin (b\pi/q)], \, \infty = [1:0]\}$.
  We define a map $\Phi_{a, b}$ from $\mathbb{P}^1$ to $\mathbb{P}^1$ as the following composition:
  \begin{align*}
    [Z_0:Z_1]
    &\mapsto
    w=\frac{(p \sin (b\pi/q))Z_0 + (q \sin (a\pi/p))Z_1}
           {(p \sin (b\pi/q))Z_0 - (q \sin (a\pi/p))Z_1}\\
    &\mapsto 2 \cos (ar\pi/p - bs\pi/q)
      - 4 \sin(ar\pi/p)\sin(bs\pi/q) \frac{w}{w-1}.
  \end{align*}
  The map $\Phi_{a, b}$ sends
  \begin{align*}
    [ q \sin (a\pi/p) : -p \sin (b\pi/q)] &\mapsto 0 \mapsto 2 \cos(ar\pi/p - bs\pi/q); \\
    [ q \sin (a\pi/p) : p \sin (b\pi/q)] &\mapsto \infty \mapsto 2 \cos(ar\pi/p + bs\pi/q); \\
    \infty=[1:0] & \mapsto 1 \mapsto \infty.
  \end{align*}
  Therefore our map $\Phi_{a, b}$ gives an isomorphism from $L_{a, b}$ to $\C_{a, b}$ in $X^{irr}(G_{p,\, q})$.
\end{proof}

\section{Adjoint Reidemeister torsion for torus knot exteriors}
We review the adjoint Reidemeister torsion for torus knot exteriors.
The adjoint Reidemeister torsion is defined for $E_{p,\, q}$ with a closed curve $\gamma$ on the boundary torus $\bd E_{p,\, q}$ and
an $\SL$-representation $\rho$ of $G_{p,\, q}=\pi_1(E_{p,\, q})$ such that the homology class $[\gamma]$ in $H_1(\bd E_{p,\, q};\Z)$ is nontrivial and $\rho$ is irreducible.
We choose $\gamma$ as a meridian $\mu$ in this paper.
The adjoint Reidemeister torsion is regarded as a function on the set of irreducible $\SL$-representations $\rho$.
More precisely, we need to choose $\mu$-regular $\SL$-representations to define the adjoint Reidemeister torsion.
However all irreducible $\SL$-representations of any torus knot group are $\mu$-regular.
We denote by $\Tor{\rho}$ the adjoint Reidemeister torsion of the torus knot exterior $E_{p,\, q}$ simply since we only consider torus knot exteriors. For the detail definition, see~\cite{Dubois1, Porti}.
The adjoint Reidemeister torsion is invariant under conjugation of $\rho$.
We think of the adjoint Reidemeister torsion as a function on the set of conjugacy classes of irreducible $\SL$-representations of $G_{p,\, q}$.

We also review the sum of the adjoint Reidemeister torsions on the level set of $\mathrm{tr}_\mu$.
The adjoint Reidemeister torsion $\Tor{\rho}$ has the constant value
\[
  \Tor{\rho}
  = \frac{pq}{16\sin^2 (\pi a/p) \sin^2(\pi b/q)}
\]
on the component $\C_{a, b}$.
According to \cite{GKY, Yoon, Yoon2, TranY},
the adjoint Reidemeister torsion for hyperbolic $3$-manifolds can be described with the Jacobian determinant of defining polynomials for the character varieties.
The irreducible part $X^{irr}(G_{p,\, q})$ is isomorphic to $Y_{p,\, q} \setminus D_{p,\, q}$ in Theorem~\ref{thm:birational}. $Y_{p,\, q} \setminus D_{p,\, q}$ is contained in the exceptional lines defined by $F_X=0$ and $F_Y=0$ where $F(X, Y)=C_p(X)-C_q(Y)$. 
We can recover the adjoint Reidemeister torsion of $E_{p,\, q}$ with $\mu$ from the Jacobian determinant of defining polynomials for $Y_{p,\, q} \setminus D_{p,\, q}$.
\begin{theorem}
  \label{thm:RtorsionHessian}
  The adjoint Reidemeister torsion $\Tor{\rho}$ on $\C_{a, b} \simeq L_{a, b}$ is expressed as
  \[
    \Tor{\rho} = \frac{-1}{4pq} \det
    \left.
      \left(\frac{\partial (F_X, F_Y)}{\partial (X, Y)} \right)
    \right|_{(X, Y)=(2\cos (a\pi/p), 2\cos (b\pi/q))}.
  \]
\end{theorem}
\begin{proof}
  By $C'_k(z)=k S_{k-1}(z)$ and $S'_{k-1}(z) = (k C_k(z)-zS_{k-1}(z))/(z^2-4)$,
  we have
  \begin{align*}
    \det
    \left.
    \left(\frac{\partial (F_X, F_Y)}{\partial (X, Y)} \right)
    \right|_{(X, Y)=(2\cos (a\pi/p), 2\cos (b\pi/q))}
    &=
      \det
      \begin{pmatrix}
        \frac{(-1)^a p^2}{-2 \sin^2 (a\pi/p)} & 0 \\
        0 & -\frac{(-1)^b q^2}{-2 \sin^2 (b\pi/q)}
      \end{pmatrix}\\
    &=-\frac{p^2 q^2}{4 \sin^2(a\pi/p) \sin^2(b\pi/q)}.
  \end{align*}
  The last equality follows from $a \equiv b \pmod{2}$.
\end{proof}

The level set of $\mathrm{tr}_\mu$ at a generic $c \in \C$ consists of $(p-1)(q-1)/2$ points.
According to the conjectures in~\cite{GKY},
we consider the sum of $(g-1)$st power of twice the adjoint Reidemeister torsion $\Tor{\rho}$ on the level set of $\mathrm{tr}_\mu$ at a generic $c \in \C$.
This sum is expressed as
\begin{align}
  \sum_{[\rho] \in \mathrm{tr}_\mu^{-1}(c)} (2\Tor{\rho})^{g-1}
  &= \sum_{\substack{0 < a < p \\ 0 < b < q \\ a \equiv b \pmod{2}}}
  \left(\frac{pq}{8\sin^2 (\pi a/p) \sin^2(\pi b/q)}\right)^{g-1} \notag \\
  &= \frac{1}{2} \sum_{\substack{0 < a < p \\ 0 < b < q}}
  \left(\frac{pq}{8\sin^2 (\pi a/p) \sin^2(\pi b/q)}\right)^{g-1}.
  \label{eq:sum_torsions}
\end{align}
The last equality follows from $\sin ( \pi b / q) = \sin ( \pi (q - b) / q)$
and $q-b$ has the different parity from $b$.

In the case of $g=0$, the sum~\eqref{eq:sum_torsions} always equals $1$ (see~\cite{TranY} for the details).
We will show that the sum~\eqref{eq:sum_torsions} for any torus knot exterior $E_{p,\, q}$ and $g>0$ turns into an integer by using Verlinde formulas.
We put the simplest example below.

\begin{example}
  The sum~\eqref{eq:sum_torsions} for the trefoil knot exterior $E_{2, 3}$ is expressed as follows:
  \begin{align*}
    \sum_{[\rho] \in \mathrm{tr}_\mu^{-1}(c)} (2\Tor{\rho})^{g-1}
    &= \left(\frac{6}{8\sin^2 (\pi/2)\sin^2(\pi/3)}\right)^{g-1}\\
    &=1
  \end{align*}
\end{example}

\section{Verlinde numbers for torus knot exteriors}
In this section, first, we introduce the modular $S$-matrix for torus knot $T_{p,\, q}$ which is a generalization of the modular $S$-matrix in $\mathrm{SU}(2)$ Wess-Zumino-Witten model. One can also find other motivations in minimal string theory~\cite{KOPSeibergShih}. 
The $S$-matrix $S=\{S_{(m,\, n),\, (r,\, s)}\}$ $(0 < m, r < p,\, 0 < n, s < q)$ in \cite[Eq.~(3.13)]{KOPSeibergShih} was defined as
\[S_{(m,\, n),\ (r,\, s)} = -\sqrt{\frac{8}{pq}} (-1)^{sm+rn} \sin\left(\frac{\pi snp}{q}\right) \sin\left(\frac{\pi rmq}{p}\right)\]
which comes from the singular points on the Chebyshev curve $C_{p,\, q}$.
Our $S$-matrix looks like a slight modification of the parametrization over the singular points as follows.
\begin{definition}
    For integers $i,\, j,\, a,\, b$ satisfying $0 < i, a < p$ and $0 < j,\, b < q$, we define the modular $S$-matrix $S=\{ S_{(i,\, j),\, (a,\, b)} \}$ by 
    \begin{equation}
        S_{(i,\, j),\, (a,\, b)}:= \sqrt{\frac{8}{pq}} \sin \left( \frac{i a \pi}{p} \right) \sin \left( \frac{j b \pi}{q} \right)
    \end{equation}
\end{definition}
\begin{remark}
  For $p=k+2,\ q=2$, this matrix $S$ is essentially equal to the modular $S$-matrix multiplied by $\sqrt{2}$
  in $\mathrm{SU}(2)$ Wess-Zumino-Witten model with level $k$,
  temporarily dropping the convention that $q$ is odd.
\end{remark}
\begin{remark}
  For $(a,\, b)=(1,\, 1)$, the number $S_{(i,\, j),\, (1,\, 1)}$ is essentially equal to the adjoint Reidemeister torsion of $T_{p,\, q}$. More precisely it holds $2\Tor{\rho} = S_{(i,\, j),\, (1,\, 1)}^{-2}$.
  This fact is closely related to the interesting program~\cite{CGK, CQW} to get a modular tensor category from a $3$-manifold based on the $3$d-$3$d correspondence \cite{TY1, TY2, DGG, CCV}.  
\end{remark}
Second, we consider an extension of the Verlinde number for a closed surface of genus $g$ with $n$-points by using the modular $S$-matrix. 
    For $n$-pairs $(a_k,\, b_k)$ $(1 \leq k \leq n)$ of two integers satisfying $0 < a_1, \dots , a_n  < p$ and $0 < b_1, \dots , b_n < q$, we define the following extension of Verlinde number for a closed surface of genus $g$ with $n$-points by
    \begin{align}
      N_g((a_1,\, b_1), \dots , (a_n,\, b_n))
      &:= \sum_{\substack{0 < i < p \\ 0 < j < q}}
      S_{(i,\, j),\, (1,\, 1)}^{2-2g}
      \prod_{k=1}^{n} \frac{S_{(i,\, j),\, (a_k,\, b_k)}}{S_{(i,\, j),\, (1,\, 1)}} \\
      &= \sum_{\substack{0 < i < p \\ 0 < j < q}}
      \left(
      \frac{pq}{8 \sin^2 (i \pi / p) \sin^2 (j \pi / q)}
      \right)^{g-1}
      \prod_{k=1}^{n}\frac{\sin(i a_k \pi / p) \sin(j b_k \pi / q)}{\sin(i\pi / p) \sin(j\pi/q)} \notag
    \end{align}
We can regard $n$-pairs of two integers as a set of multiplicities $n_{a, b}$ of indices $(a, b)$ for $0 < a < p$ and $0 < b < q$ such that $n_{1,\, 1} + \dots + n_{p-1,\, q-1} = n$.
We also define the Verlinde numbers of a torus knot exterior
by adding some coefficients of multiplicity in the extension of Verlinde number for a closed surface above.
Our Verlinde numbers of a torus knot exterior will show that the sum~\eqref{eq:sum_torsions} of the adjoint Reidemeister torsions turns into an integer. These numbers are generalizations of the Verlinde numbers in $\SU$ Wess-Zumino-Witten models. Various definitions and nice properties are known for the Verlinde numbers in Wess-Zumino-Witten models. Good references include \cite{BT, J, S, T, Z}. In particular, the proof of the integrality of the generalized Verlinde numbers for torus knots in this paper refers to the discussion of Verlinde numbers in \cite{Z}.
\begin{definition}
  We define the Verlinde number $d(g, \bm{n})$ of the $(p,\, q)$-torus knot exterior
  for $(g,\, \bm{n} = \sum_{0<a<p,\, 0<b<q} n_{a,\, b} \bm{\ell}_{a,\, b}) \in \Z_{\geq 0} \times S_{p,\, q}$\, ($S_{p,\, q} = \Z_{\geq 0} \, \bm{\ell}_{1,\, 1} + \cdots + \Z_{\geq 0} \, \bm{\ell}_{p-1,\, q-1}$) as
  \begin{equation}
    d(g, \bm{n})
    =\left(\frac{1}{4}\right)^{g-1}
    \left(\frac{1}{2}\right)^{|\bm{n}|}
    \sum_{\substack{0 < i < p \\ 0 < j < q}}
      \left(\frac{p}{2\sin^2 \frac{\pi i}{p}} \frac{q}{2\sin^2 \frac{\pi j}{q}}\right)^{g-1}
    \prod_{\substack{0 < a < p \\ 0 < b < q}}
      \left(\frac{\sin \frac{\pi i a}{p}}{\sin \frac{\pi i}{p}}\frac{\sin \frac{\pi j b}{q}}{\sin \frac{\pi j}{q}}\right)^{n_{a,\, b}}
  \end{equation}
  where $|\bm{n}| = \sum_{0<a<p,\, 0<b<q} n_{a,\, b}$ (called the \textit{weight of $\bm{n}$}).
\end{definition}

By definition, the Verlinde number $d(g, \bm{0})$ satisfies
\[
  d(g, \bm{0})=
  2 \sum_{[\rho] \in \mathrm{tr}_\mu^{-1}(c)} (\Tor{\rho})^{g-1}.
\]

\begin{remark}
  In the case of $g=0$,
  the sum in the right hand side above turns into $2$ for all torus knots (see~\cite{TranY}).
  Hence it holds 
  \[d(0, \bm{0})=4.\]
  It is easy to see that the Verlinde number for $g=1$ and $\bm{n}=\bm{0}$ turns out to be
  \[d(1, \bm{0})=(p-1)(q-1).\]

\end{remark}

The sum~\eqref{eq:sum_torsions} of the adjoint Reidemeister torsions equals $2^{g-2} d(g, \bm{0})$.
We will show that $2^{g-2} d(g, \bm{0})$ is an integer
according to the following recurrence formulas of the Verlinde number $d(g, \bm{n})$.

\section{Fusion rules and the sum of the adjoint Reidemeister torsion}
To show that $2^{g-2} d(g, \bm{0})$ is an integer,
we provide the recurrence formulas of our Verlinde numbers $d(g, \bm{n})$ which are often referred to as the \textit{Fusion rules}.
\begin{proposition}[Fusion rules]
  \label{prop:Fusionrules}
  The Verlinde numbers $d(g, \bm{n})$ of the $(p,\, q)$-torus knot exterior satisfy the following recurrence formulas:
  \begin{enumerate}
  \item \label{item:fusionI}
    $\displaystyle \sum_{\substack{0<a<p \\ 0<b<q}}
    d(g, \bm{n}+2\bm{\ell}_{a,\, b})=d(g+1, \bm{n})$
  \item \label{item:fusionII}
    $\displaystyle \sum_{\substack{0<a<p \\ 0<b<q}}
    d(g, \bm{n}+\bm{\ell}_{a,\, b}) d(g', \bm{n}'+\bm{\ell}_{a,\, b})=d(g+g', \bm{n}+\bm{n}')$
  \end{enumerate}
\end{proposition}
\begin{lemma}
  \label{lem:Zagier}
  For $\zeta_1, \zeta_2 \in \C$ such that $\zeta_i ^{2k} =1$ and $0 < \mathrm{arg}(\zeta_i) < \pi$ ($i=1, 2$),
  it holds that
  \[\sum_{0<a<k}S_{a-1}(\zeta_1 + \zeta_1^{-1})S_{a-1}(\zeta_2 + \zeta_2^{-1}) =
  \begin{cases}
    \frac{-2k}{(\zeta_1-\zeta_1^{-1})^2} & \text{if}\quad  \zeta_1 = \zeta_2 \\
    0 & \text{otherwise}
  \end{cases}
  \]
  where $S_{a-1}(z)$ is the Chebyshev polynomial of the second kind.

  In particular, we have
  \[\sum_{0 < a < k}
  \left(\frac{\sin (\pi j a / k)}{\sin (\pi j /k)}\right)^2 = \frac{k}{2\sin^2(\pi j / k)}\]
  where j is an integer such that $0 < j < k$.
\end{lemma}
\begin{proof}
  If $\zeta \in \C$ satisfies $\zeta^{2k} = 1$ and $\zeta \not = 1$, then
  $\zeta^{2k-1} + \cdots +\zeta^{k+1} + \zeta^k + \zeta^{k-1} + \cdots + \zeta + 1 = 0$.
  We can rewrite this equality as
  \begin{align}
    \sum_{0 < a < k}\zeta^{2k-a} + \zeta^k + \sum_{0 < a < k} \zeta^a + 1 &= 0 \notag \\
    \sum_{0 < a < k}\zeta^{-a} + \zeta^k + \sum_{0 < a < k} \zeta^a + 1  &= 0 \notag \\
    \sum_{0 < a < k}\zeta^{-a} + \sum_{0 < a < k} \zeta^a
    &=
    \begin{cases}
    -2 & \text{if}\quad \zeta^k = 1 \\
    0 & \text{if}\quad \zeta^k = -1
    \end{cases}
    \label{Eq:sum_zeta}
  \end{align}
  which implies the lemma above together with
  \[S_{a-1}(\zeta + \zeta^{-1}) = \frac{\zeta^a - \zeta^{-a}}{\zeta - \zeta^{-1}}.\]
\end{proof}

\begin{proof}[Proof of Proposition~\ref{prop:Fusionrules}]
  We will see that the left hand sides turn into the right hand sides as follows.

  \noindent \eqref{item:fusionI}  
  By definition, we can write $d(g, \bm{n}+2\bm{\ell}_{a,\, b})$ as
  \begin{align*}
    &d(g, \bm{n}+2\bm{\ell}_{a,\, b})\\
    &=
    \left(\frac{1}{4}\right)^{g-1}
    \left(\frac{1}{2}\right)^{|\bm{n}|+2} \cdot 
    \sum_{\substack{0 < i < p \\ 0 < j < q}}
      \left(\frac{p}{2\sin^2 \frac{\pi i}{p}} \frac{q}{2\sin^2 \frac{\pi j}{q}}\right)^{g-1}
      \left(\frac{\sin \frac{\pi i a}{p}}{\sin \frac{\pi i}{p}}
      \frac{\sin \frac{\pi j b}{q}}{\sin \frac{\pi j}{q}}\right)^2
      \prod_{\substack{0 < a < p \\ 0 < b < q}}      
      \left(\frac{\sin \frac{\pi i a}{p}}{\sin \frac{\pi i}{p}}
      \frac{\sin \frac{\pi j b}{q}}{\sin \frac{\pi j}{q}}\right)^{n_{a, b}}.
  \end{align*}
  By Lemma~\ref{lem:Zagier},  we can also rewrite the summation of $d(g, \bm{n}+2\bm{\ell}_{a,\, b})$ as
  \begin{align*}
    &\sum_{\substack{0<a<p \\ 0<b<q}}
    d(g, \bm{n}+2\bm{\ell}_{a,\, b})\\
    &=
    \left(\frac{1}{4}\right)^{g}
    \left(\frac{1}{2}\right)^{|\bm{n}|} \\
    &\quad \cdot
    \sum_{\substack{0 < i < p \\ 0 < j < q}}
      \left(\frac{p}{2\sin^2 \frac{\pi i}{p}} \frac{q}{2\sin^2 \frac{\pi j}{q}}\right)^{g-1}
      \frac{p}{2 \sin^2 (\pi i / p)}
      \frac{q}{2 \sin^2 (\pi j / q)}
      \prod_{\substack{0 < a < p \\ 0 < b < q}}      
      \left(\frac{\sin \frac{\pi i a}{p}}{\sin \frac{\pi i}{p}}
      \frac{\sin \frac{\pi j b}{q}}{\sin \frac{\pi j}{q}}\right)^{n_{a,\, b}}\\
    &= d(g+1, \bm{n}).      
  \end{align*}

  \noindent\eqref{item:fusionII}  
  The product $d(g, \bm{n}+\bm{\ell}_{a,\, b})d(g', \bm{n'}+\bm{\ell}_{a,\, b})$ turns into
  \begin{align*}
    &\left(\frac{1}{4}\right)^{g+g'-1}
    \left(\frac{1}{2}\right)^{|\bm{n}|+|\bm{n'}|}  \\
    &\quad \cdot
    \left(
    \sum_{\substack{0 < i < p \\ 0 < j < q}}
    \left(\frac{p}{2\sin^2 \frac{\pi i}{p}} \frac{q}{2\sin^2 \frac{\pi j}{q}}\right)^{g-1}
    S_{a-1}(\zeta^i + \zeta^{-i}) S_{b-1}(\eta^j + \eta^{-j})
    \prod_{\substack{0 < a < p \\ 0 < b < q}}
    \left(\frac{\sin \frac{\pi i a}{p}}{\sin \frac{\pi i}{p}}
    \frac{\sin \frac{\pi j b}{q}}{\sin \frac{\pi j}{q}}\right)^{n_{a,\, b}}
    \right)  \\
    &\quad \cdot
    \left(
    \sum_{\substack{0 < i < p \\ 0 < j < q}}
    \left(\frac{p}{2\sin^2 \frac{\pi i}{p}} \frac{q}{2\sin^2 \frac{\pi j}{q}}\right)^{g'-1}
    S_{a-1}(\zeta^i + \zeta^{-i}) S_{b-1}(\eta^j + \eta^{-j})
    \prod_{\substack{0 < a < p \\ 0 < b < q}}
    \left(\frac{\sin \frac{\pi i a}{p}}{\sin \frac{\pi i}{p}}
    \frac{\sin \frac{\pi j b}{q}}{\sin \frac{\pi j}{q}}\right)^{n'_{a,\, b}}
    \right)
  \end{align*}
  where $\zeta = e^{\pi \sqrt{-1}/p}$ and $\eta = e^{\pi \sqrt{-1}/q}$.

  Lemma~\ref{lem:Zagier} shows that the summation of $d(g, \bm{n}+\bm{\ell}_{a,\, b})d(g', \bm{n'}+\bm{\ell}_{a,\, b})$ turns out to be
  \begin{align*}
    &\sum_{\substack{0<a<p \\ 0<b<q}} d(g, \bm{n}+\bm{\ell}_{a,\, b})d(g', \bm{n'}+\bm{\ell}_{a,\, b}) \\
    &=\left(\frac{1}{4}\right)^{g+g'-1}
    \left(\frac{1}{2}\right)^{|\bm{n}|+|\bm{n'}|}  \\
    &\quad \cdot
    \sum_{\substack{0 < i < p \\ 0 < j < q}}
      \left(\frac{p}{2\sin^2 \frac{\pi i}{p}} \frac{q}{2\sin^2 \frac{\pi j}{q}}\right)^{g+g'-2}
      \frac{p}{2 \sin^2 (\pi i / p)}
      \frac{q}{2 \sin^2 (\pi j / q)}
      \prod_{\substack{0 < a < p \\ 0 < b < q}}      
      \left(\frac{\sin \frac{\pi i a}{p}}{\sin \frac{\pi i}{p}}
      \frac{\sin \frac{\pi j b}{q}}{\sin \frac{\pi j}{q}}\right)^{n_{a,\, b}+n'_{a,\, b}}\\
    &=d(g+g', \bm{n}+\bm{n'}).    
  \end{align*}
\end{proof}

We also carry out the initial values of the Verlinde numbers as follows.
\begin{proposition}[Initial values of $d(g, \bm{n})$]
  \label{prop:initialvalues}
  The Verlinde numbers $d(0, \bm{n})$ of the $(p,\, q)$-torus knot exterior for $|\bm{n}| (=n_{1,\, 1} + \cdots + n_{p-1,\, q-1}) = 1, 2, 3$ take the following values:
  \begin{enumerate}
  \item \label{item:initial_i}
    $d(0, \bm{\ell}_{a,\, b}) =
    \begin{cases}
      2 & \text{if $(a,\, b) = (1,\, 1)$} \\
      0 & \text{otherwise}
    \end{cases}$
  \item \label{item:initial_ii}
    $d(0, \bm{\ell}_{a,\, b}+\bm{\ell}_{c,\, d}) =
    \begin{cases} 1 & \text{if $(a,\, b)=(c,\, d)$} \\
      0 & \text{otherwise}\end{cases}$
  \item \label{item:initial_iii}
    \label{item:d_0_n3} $d(0, \bm{\ell}_{a,\, b}+\bm{\ell}_{c,\, d}+\bm{\ell}_{e,\, f})
    = \begin{cases}
      1/2 & \text{if $2\mathrm{max}\{a, c, e\} < a+c+e < 2p$, $a+c+e$ odd} \\
      & \text{and $2\mathrm{max}\{b, d, f\} < b+d+f < 2q$, $b+d+f$ odd} \\
      0 & \text{otherwise}
      \end{cases}$ 
  \end{enumerate}
\end{proposition}
\begin{proof}[Proof of Proposition~\ref{prop:initialvalues}]
  \noindent \eqref{item:initial_i}
  If $(a,\, b) = (1,\, 1)$, then we have $d(0, \bm{\ell}_{1, 1})=(1/2)d(0, \bm{0})=2$.
  In the case of $(a,\, b) \not = (1,\, 1)$, the Verlinde number $d(0, \bm{\ell}_{a,\, b})$ is expressed as
  \begin{align*}
    d(0, \bm{\ell}_{a,\, b})
    &= 2 \sum_{\substack{0 < i < p \\ 0 < j < q}} \frac{2\sin\frac{\pi i}{p} 2\sin\frac{\pi j}{q}}{pq}
    \sin\frac{\pi i a}{p} \sin \frac{\pi j b}{q} \\
    &= \frac{1}{2pq}
    \sum_{0 < i < p} (\zeta^i - \zeta^{-i})(\zeta^{ia} - \zeta^{-ia})
    \sum_{0 < j < q} (\eta^j - \eta^{-j})(\eta^{jb} - \eta^{-jb}) \\
    &= \frac{1}{2pq}
    \sum_{0 < i < p} (\zeta^{i(a+1)} + \zeta^{-i(a+1)} - (\zeta^{i(a-1)} + \zeta^{-i(a-1)}))  \\
    &\hphantom{\frac{-2}{pq}}\quad
    \cdot \sum_{0 < j < q} (\eta^{j(b+1)} + \eta^{-j(b+1)} - (\eta^{j(b-1)} + \eta^{-j(b-1)})) 
  \end{align*}
  where $\zeta=e^{\pi \sqrt{-1}/p}$ and $\eta=e^{\pi \sqrt{-1}/q}$.
  If $a \not = 1$ (resp. $b \not = 1$), then the sum of $\zeta$ (resp. $\eta$) vanishes by Eq.~\eqref{Eq:sum_zeta}
  since the exponents of $\zeta$ (resp. $\eta$) have the same parity.

  \noindent \eqref{item:initial_ii}
  We can express $d(0, \bm{\ell}_{a,\, b} + \bm{\ell}_{c,\, d})$ as
  \begin{align*}
    d(0, \bm{\ell}_{a,\, b} + \bm{\ell}_{c,\, d})
    &= \frac{4}{pq}
    \sum_{\substack{0 < i < p \\ 0 < j < q}}
    \sin \frac{\pi i a}{p} \sin \frac{\pi i c}{p} \sin\frac{\pi j b}{q} \sin \frac{\pi j d}{q} \\
    &= \frac{1}{4pq}
    \sum_{0 < i < p} (\zeta^{ia} - \zeta^{-ia})(\zeta^{ic} - \zeta^{-ic})
    \sum_{0 < j < q} (\eta^{jb} - \eta^{-jb})(\eta^{jd} - \eta^{-jd}) \\
    &= \frac{1}{4pq}
    \sum_{0 < i < p} (\zeta^{i(a+c)} + \zeta^{-i(a+c)} - (\zeta^{i(a-c)} + \zeta^{-i(a-c)})) \\
    &\phantom{\frac{4}{pq}} \quad \cdot
    \sum_{0 < j < q} (\eta^{j(b+d)} + \eta^{-j(b+d)} - (\eta^{j(b-d)} + \eta^{-j(b-d)}))
  \end{align*}
  where $\zeta=e^{\pi \sqrt{-1}/p}$ and $\eta=e^{\pi \sqrt{-1}/q}$.
  Then Eq.~\eqref{Eq:sum_zeta} shows that this summation equals $1$ if $(a, b)=(c, d)$ or $0$ otherwise.

  \noindent\eqref{item:initial_iii}
  Set $\zeta=e^{\pi \sqrt{-1}/p}$ and $\eta=e^{\pi \sqrt{-1}/q}$.
  We can also express $d(0, \bm{\ell}_{a,\, b} + \bm{\ell}_{c,\, d} + \bm{\ell}_{e,\, f})$ as
  \begin{align*}
    &d(0, \bm{\ell}_{a,\, b} + \bm{\ell}_{c,\, d} + \bm{\ell}_{e,\, f}) \\
    &=\frac{1}{2}
    \frac{-1}{2p}
    \sum_{0 < i < p} \frac{(\zeta^{ia} - \zeta^{-ia})(\zeta^{ic} - \zeta^{-ic})(\zeta^{ie} - \zeta^{-ie})}{(\zeta^{i} - \zeta^{-i})} \\
    &\hphantom{\frac{1}{2}}\quad \cdot \frac{-1}{2q}
    \sum_{0 < j < q} \frac{(\eta^{jb} - \eta^{-jb})(\eta^{jd} - \eta^{-jd})(\eta^{jf} - \eta^{-jf})}{(\eta^{j} - \eta^{-j})}\\
    &=\frac{1}{2}
    \frac{-1}{2p}
    \sum_{0 < i < p} (\zeta^{i(a+c)} + \zeta^{-i(a+c)} - \zeta^{i(a-c)} - \zeta^{-i(a-c)})
    (\zeta^{i(e-1)} + \zeta^{i(e-3)} + \cdots + \zeta^{-i(e-3)} + \zeta^{-i(e-1)}) \\
    &\cdot \frac{-1}{2q}
    \sum_{0 < j < q} (\eta^{j(b+d)} + \eta^{-j(b+d)} - \eta^{j(b-d)} - \eta^{-j(b-d)})
    (\eta^{j(f-1)} + \eta^{j(f-3)} + \dots + \eta^{-j(f-3)} + \eta^{-j(f-1)}) \\
    &=\frac{1}{2}
    \frac{1}{2p}\left(
    \sum_{k=1}^e \sum_{0 < i < p} (\zeta^{i(a-c)} + \zeta^{-i(a-c)}) \zeta^{i(e - (2k-1))}
    - \sum_{k=1}^e \sum_{0 < i < p} (\zeta^{i(a+c)} + \zeta^{-i(a+c)}) \zeta^{i(e - (2k-1))} \right) \\
    &\quad \cdot
    \frac{1}{2q}\left(
    \sum_{k=1}^f \sum_{0 < j < q} (\eta^{j(b-d)} + \eta^{-j(b-d)}) \eta^{j(f - (2k-1))}
    - \sum_{k=1}^f \sum_{0 < j < q} (\eta^{j(b+d)} + \eta^{-j(b+d)}) \eta^{j(f - (2k-1))} \right).
  \end{align*}
  We focus on the sums of $\zeta$.
  Each sum of $\zeta$ vanishes if $a-c+e \equiv a+c+e \equiv 0 \pmod{2}$ by Eq.~\eqref{Eq:sum_zeta} respectively. In the case of $a-c+e \equiv a+c+e \equiv 1 \pmod{2}$,
  we have
  \begin{align*}
  &\sum_{k=1}^e \sum_{0 < i < p} (\zeta^{i(a-c+e - (2k-1))} + \zeta^{-i(a-c+e - (2k-1))})\\
  &=\begin{cases}
    -2e+2p & \text{if $\exists k$ such that $a-c+e-(2k-1)=0$ } \\
    &  \Leftrightarrow |a-c|<e ; \\
    -2e & \text{otherwise}
  \end{cases}
  \end{align*}
  and
  \begin{align*}
  &\sum_{k=1}^e \sum_{0 < i < p} (\zeta^{i(a+c+e - (2k-1))} + \zeta^{-i(a+c+e - (2k-1))})\\
  &=\begin{cases}
    -2e+2p & \text{if $\exists k$ such that $a+c+e-(2k-1)=0$ or $2p$} \\
    & \text{$\Leftrightarrow a+c<e$ or $a+c+e>2p$}; \\
  -2e & \text{otherwise}
  \end{cases}
  \end{align*}
  Note that the case of $|a-c| \ge e$ and $a+c<e$ or $a+c+e>2p$ does not occur.
  Summarizing, we can see
  \begin{align*}
    &\frac{1}{2p}\left(
    \sum_{k=1}^e \sum_{0 < i < p} (\zeta^{i(a-c)} + \zeta^{-i(a-c)}) \zeta^{i(e - (2k-1))}
    - \sum_{k=1}^e \sum_{0 < i < p} (\zeta^{i(a+c)} + \zeta^{-i(a+c)}) \zeta^{i(e - (2k-1))} \right)\\
    &=
    \begin{cases}
      1 & \text{if $2\mathrm{max}\{a, c, e\} < a+c+e < 2p$, $a+c+e$ odd}; \\
      0 & \text{otherwise}.
    \end{cases}
  \end{align*}
  Similar arguments apply to the sum of $\eta$ which proves the initial value~\eqref{item:initial_iii}.
\end{proof}

We can compute the Verlinde numbers $d(g, \bm{0})$ for any $g>0$ by the initial values above and the Fusion rules.
\begin{proposition}
  \label{prop:Verlinde}
  Let $d(g, \bm{n})$ denote the Verlinde number of the $(p,\, q)$-torus knot exterior.
  The Verlinde number $d(g, \bm{0})$ is in $(1/2)^{g-2}\Z$.
\end{proposition}
Then we obtain the integrality of the sum~\eqref{eq:sum_torsions} of the adjoint Reidemeister torsions since the sum~\eqref{eq:sum_torsions} equals $2^{g-2} d(g, \bm{0})$.
\begin{theorem}
  \label{thm:integerness}
  The sum~\eqref{eq:sum_torsions} of the adjoint Reidemeister torsions for any $(p,\, q)$-torus knot exterior and $g \geq 0$ is an integer, that is, for every generic $c \in \C$
  \[\sum_{[\rho] \in \mathrm{tr}_\mu^{-1}(c)} (2\Tor{\rho})^{g-1} \in \Z\]
  for $\forall g \geq 0$.
\end{theorem}

The remainder of this section will be devoted to the proof of Proposition~\ref{prop:Verlinde}.
\begin{proof}[Proof of Proposition~\ref{prop:Verlinde}]
  By the Fusion rules, we can rewrite $d(g, \bm{n})$ as
  \begin{align*}
    d(g, \bm{0})
    &= \sum_{\substack{0< a_1<p \\ 0<b_1<q}}
    d(g-1, \bm{\ell}_{a_1,\, b_1})d(1, \bm{\ell}_{a_1,\, b_1}) \\
    &= \sum_{\substack{0< a_1<p \\ 0<b_1<q}}
    \left(\sum_{\substack{0< a_2<p \\ 0<b_2<q}}
    d(g-2, \bm{\ell}_{a_1,\, b_1}+\bm{\ell}_{a_2,\, b_2})
    d(1, \bm{\ell}_{a_2,\, b_2})\right) d(1, \bm{\ell}_{a_1,\, b_1}) \\
    &= \sum_{\substack{0< a_1<p \\ 0<b_1<q}} \cdots
    \sum_{\substack{0< a_g<p \\ 0<b_g<q}}
    d(0, \bm{\ell}_{a_1,\, b_1} + \cdots + \bm{\ell}_{a_g,\, b_g})
    d(1, \bm{\ell}_{a_g,\, b_g}) \cdots d(1,\bm{\ell}_{a_1,\, b_1}).
  \end{align*}
  Our claim follows from the following lemmas~\ref{lemma:d_1} and~\ref{lemma:d_0}
\end{proof}

\begin{lemma}
  \label{lemma:d_1}
  The Verlinde number $d(1, \bm{\ell}_{a,\, b})$ of the $(p,\, q)$-torus knot exterior
  satisfies
  \begin{equation}
    d(1, \bm{\ell}_{a,\, b})
    = \begin{cases}
      (p-a)(q-b)/2 & \text{if $a$ and $b$ are odd} \\
      0 & \text{otherwise}
    \end{cases}
  \end{equation}
  In particular, $d(1, \bm{\ell}_{a,\, b})$ is an integer.
\end{lemma}
\begin{proof}
  We can rewrite $d(1, \bm{\ell}_{a,\, b})$ as 
  \[
    d(1, \bm{\ell}_{a,\, b})
    = \sum_{\substack{0<c<p \\ 0<d<q}} d(0, \bm{\ell}_{a,\, b}+2\bm{\ell}_{c,\, d}).
  \]
  It follows from Proposition~\ref{prop:initialvalues}~\eqref{item:d_0_n3} that
  $d(0, \bm{\ell}_{a,\, b}+2\bm{\ell}_{c,\, d})$ satisfies
  \[
    d(0, \bm{\ell}_{a,\, b}+2\bm{\ell}_{c,\, d})
    = \begin{cases}
      1/2 & \text{if $a$ odd, $a/2 < c < p - a/2$ and $b$ odd, $b/2 < d < q - b/2$};\\
      0 & \text{otherwise}
      \end{cases}
   \]
   which proves our lemma.
\end{proof}

\begin{lemma}
  \label{lemma:d_0}
  $d(0, \bm{\ell}_{a_1,\, b_1} + \cdots + \bm{\ell}_{a_g,\, b_g}) \in (1/2)^{g-2}\Z$
\end{lemma}
\begin{proof}
  The cases of $g=1, 2$ and $3$ follow from Proposition~\ref{prop:initialvalues}.
  For $g>3$, we can rewrite $d(0, \bm{\ell}_{a_1,\, b_1} + \cdots + \bm{\ell}_{a_g,\, b_g})$ as
  \[
    d(0, \bm{\ell}_{a_1,\, b_1} + \cdots + \bm{\ell}_{a_g,\, b_g})
    = \sum_{\substack{0<a<p \\ 0<b<q}}
    d(0, \bm{\ell}_{a_1,\, b_1} + \cdots + \bm{\ell}_{a_{g-2},\, b_{g-2}}
    + \bm{\ell}_{a,\, b})
    d(0, \bm{\ell}_{a_{g-1},\, b_{g-1}} + \bm{\ell}_{a_g,\, b_g} + \bm{\ell}_{a,\, b}).
  \]
  Our claim follows from the induction on $g$.
\end{proof}

\section{Examples}
Classical Verlinde numbers are closely related to the cases of $(2, q)$-torus knots.
The Chebyshev curves of type $(2, q)$ are illustrated as in Figure~\ref{fig:Chebyshev_2q}.
\begin{figure}[ht]
  \includegraphics[width=12cm]{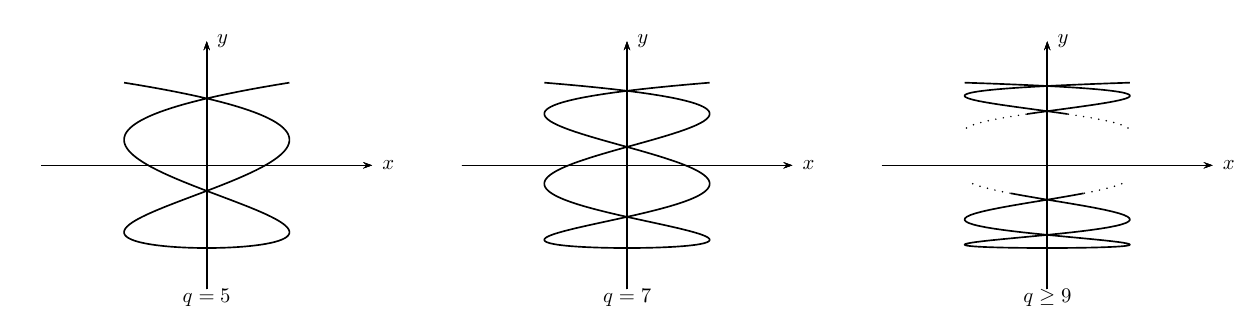}
  \caption{Chebyshev curves of type $(2, q)$ with $q > 3$}
  \label{fig:Chebyshev_2q}
\end{figure}
We also touch a relation between our Verlinde numbers and the classical Verlinde numbers.
\begin{example}[the $(2, q)$-torus knot exteriors]
  The Verlinde number $d(g, \bm{0})$ of the $(2,q)$-torus knot is expressed as
  \begin{equation}
    \label{eq:Verlinde_2q}
    d(g, \bm{0})
    =\left(\frac{1}{4}\right)^{g-1}
      \sum_{0<j<q} \left(\frac{q}{2\sin^2 \frac{\pi j}{q}}\right)^{g-1}.
  \end{equation}
  The classical Verlinde number $(q/2)^{g-1}\sum_{0<j<q} \sin^{2-2g} (\pi j / q)$ appears in the right hand side of~\eqref{eq:Verlinde_2q}.
\end{example}
We have seen that our Verlinde number $d(g, \bm{0})$ for the $(2, q)$-torus knot is an element in $(1/2)^{g-2}\Z$. We also have the following corollary.
\begin{corollary}
  Theorem~\ref{thm:integerness} shows that 
  the classical Verlinde number takes its value in $2^g \Z$, that is, it holds that 
  \[\left(\frac{q}{2}\right)^{g-1}\sum_{0<j<q} \sin^{2-2g} \frac{\pi j}{q} \in 2^g \Z\]
  for odd $q$ and $\forall g \ge 0$.
\end{corollary}


\end{document}